\documentclass[reqno]{amsart}

\usepackage{amssymb,amscd,color}
\usepackage{pdfsync}
\usepackage{graphics}
\usepackage{graphicx}

\theoremstyle{plain}

\newtheorem{theorem}{Theorem}[section]

\newtheorem{lemma}[theorem]{Lemma}
\newtheorem{proposition}[theorem]{Proposition}
\newtheorem{corollary}[theorem]{Corollary}

\theoremstyle{definition}

\newcommand{\N}{\hbox{\ensuremath{\mathbb{N}}}}
\newcommand{\Z}{\hbox{\ensuremath{\mathbb{Z}}}}

\newcommand{\R}{\hbox{\ensuremath{\mathbb{R}}}}
\newcommand{\C}{\hbox{\ensuremath{\mathbb{C}}}}

\newcommand{\rk}{\mathrm{rank\ }}

\newcommand{\sinc}{\text{sinc}}
\newcommand{\Ker}{\text{Ker}}

\newcommand{\beq}{\begin {equation}}
\newcommand{\eeq}{\end{equation}}

\newcommand{\esssup}{{\rm ess\,sup}\ }

\numberwithin{equation}{section}
\begin{document}


\title[Uncertainty Principles]{Uncertainty Principles in Finitely generated Shift-Invariant Spaces with additional invariance}


\author[Romain Tessera]{Romain~Tessera}
\address[Romain Tessera]{Department of Mathematics\\
Ecole Normale Superieure de Lyon\\ France}
\email
{romtessera@gmail.com}

\author[Haichao Wang]{Haichao~Wang}
\address[Haichao Wang]{Department of Mathematics\\
University of California\\ Davis, CA 95616}
\email 
{hchwang@ucdavis.edu}




\maketitle

\begin{abstract}
We consider finitely generated shift-invariant spaces (SIS) with additional invariance in $L^2(\R^d)$. We prove that if the generators and their translates form a frame, then they must satisfy some stringent restrictions on their behavior at infinity.  Part of this work (non-trivially) generalizes recent results obtained in the special case of a principal shift-invariant spaces in $L^2(\R)$ whose generator and its translates form a Riesz basis.
\end{abstract}


\section{Introduction}

Finitely generated shift-invariant spaces have been widely used in  approximation theory,  numerical analysis, sampling theory and
wavelet theory ( see e.g., \cite{ACHM07,AG01, AST05, Bow00,  CS07, DDR94, RS95, UB00} and the references therein). 
Shift-invariant spaces with additional invariance have been studied in the context of
wavelet analysis and sampling theory \cite{AKTW11, CS03,HL09, Web00}, and have been given a complete algebraic description  in \cite {ACHKM10} for $L^2(\R)$ and in \cite {ACP10} for $L^2(\R^d)$. As a tool for showing our main results, we will prove a slightly different but useful characterization. 
 
It is well-known that the Paley-Wiener space $PW$ is translation-invariant. Moreover, Shannon's sampling theorem easily implies that $PW$ is principal, i.e. generated by the single function $\sinc$. It turns out that the fact that $\sinc$ is non-integrable is not a coincidence. Actually, for a principal shift-invariant space in $L^2(\R)$ which is translation-invariant, any  frame generator is non-integrable (see for instance \cite{ASW11}). This observation holds in any dimension. Indeed, the Fourier transform $\widehat{\phi}$ of a frame generator has to satisfy for a.e. $\omega\in \R^d$
$$C^{-1}1_A(\omega)\leq |\widehat{\phi}(\omega)|\leq C 1_{A}(\omega) $$ 
for some $C\geq 1$ and some finite measure subset $A$. In particular, $\widehat{\phi}$ is not continuous. Such a condition also prevents $\widehat{\phi}$ from being in the Sobolev space\footnote{i.e. $\int|\phi(x)|^2(1+|x|)dx=\infty$. } $H^{\frac 12}(\R^d)$ (see \cite{KW99}), whereas it belongs to $H^{\frac 12-\epsilon}(\R^d)$ for every $\epsilon>0$ when $A$ is a Euclidean  ball of positive radius.

Our first result is a straightforward  generalisation of this fact to shift-invariant spaces generated by several functions.
\begin{theorem}\label{notranslation}
Let $\Lambda$ be a lattice in $\R^d$.
If a finitely generated $\Lambda$-invariant space of $L^2(\R^d)$ is translation-invariant, then at least one of its frame generators has a non continuous Fourier transform, and in particular is not in $L^1$. Moreover this generator satisfies $\int|\phi(x)|^2(1+|x|)dx=\infty$.  
\end{theorem}
 The slow spatial-decay of the generators of shift-invariant spaces that are also translation-invariant is a disadvantage for the numerical implementation of some analysis and processing algorithms. This is a motivation for considering instead shift-invariant spaces that are only $\frac 1n\Z$-invariant and hoping the generators will have better time-frequency localization. Indeed, it was shown in \cite{ASW11} that for every $n$, one can construct a principal shift-invariant space with an orthonormal generator which is in $L^1(\R)$ although the space is $\frac 1n\Z$-invariant.
However, there are still obstructions if we require more regularity on the Fourier transform of the generators. Asking for fractional differentiability yields the following Balian-Low type obstructions (see \cite{BCGP03} and the reference therein). Compare \cite[Theorem 1.2]{ASW11}.
\begin{theorem}\label{epsilonthm}
Let $\Lambda<\Gamma$ be two lattices\footnote{The notation $\Lambda<\Gamma$ is standard to denote subgroups.} in $\R^d$. 
Suppose $\phi_i\in L^2(\R^d)$ are such that $\{\phi_i(\cdot+\lambda)| \lambda\in \Lambda, i=1, \dots, r\}$ forms a frame for the closed subspace $V^{\Lambda}(\Phi)$ spanned by these functions. Let $\rho$ ($\leq r$) be the minimal number of generators of  $V^{\Lambda}(\Phi)$. Assume that $[\Gamma:\Lambda]$ is not a divisor of $\rho$, and suppose that  $V^{\Lambda}(\Phi)$ is $\Gamma$-invariant. Then there exists $i_0\in\{1, \dots, r\}$ such that 
$$\int_{\R^d} |\phi_{i_0}(x)|^2 |x|^{d+\epsilon} dx=+\infty$$
for all $\epsilon>0$. I.e.  $\widehat{\phi}_{i_0}$ is not in $H^{\frac d2+\epsilon}(\R^d)$. 
\end{theorem}

One can also ask for a combinaison of regularity, namely continuity and of control of the decay at infinity. In this spirit, we obtain the following result (compare \cite[Theorem 1.3]{ASW11}).  
\begin{theorem}\label{pointwise}
Keep the same assumptions as in the last theorem, and suppose moreover that $\widehat{\phi}_i$ is continuous for every $i=1, \dots, r$. Then there exits $i_0\in\{1, \dots, r\}$ such that $\omega^{\frac d2+\epsilon} \widehat\phi_{i_0}(\omega)\not\in L^\infty(\R^d)$ for any $\epsilon>0$, i.e.,
\begin{equation*}
\esssup_{\omega\in\mathbb R} |\widehat\phi_{i_0}(\omega)| |\omega|^{\frac d2+\epsilon} =+\infty.
\end{equation*}
\end{theorem}
For $d=1$, the exponent is sharp up to the $\epsilon$ in both theorems but it does not seem to be the case for larger $d$. We actually conjecture that the right exponent should be the same for all dimensions.

Notice that the condition $[\Gamma:\Lambda]\nmid \rho$ in the above theorems is essential. 
Indeed all regularity constraints trivially disappear when $\rho=k[\Gamma:\Lambda]$ for any integer $k>0$. To see why, start with some $\Gamma$-invariant space generated by exactly $k$ orthogonal generators $\phi_1,\ldots, \phi_k$, with --say-- smooth and compactly generated Fourier transforms. Then note that as a $\Lambda$-invariant space,   $V^{\Gamma}(\phi_1,\ldots, \phi_k)$ is generated by the orthogonal generators  $\phi_i(\cdot-f)$ for $f\in F$ and $i=1,\ldots, k$, where $F$ is a section of $\Gamma/\Lambda$ in $\Gamma$.

The previous results state that under additional invariance there always exists at least one (frame) generator whose Fourier transform has poor regularity. One may wonder if at least some generators can be chosen with good properties. We do not know what the optimal proportion of good generators should be. The following proposition gives a lower bound on the number of good generators. Observe that this bound gets worse with the dimension.

\begin{proposition}\label{prop:regular}
For every $d\geq 1$, and every $k\in \N$ there exists an SIS $V(\Phi)$ in $L^2(\R^d)$ generated by an orthonormal basis $\Phi$ consisting of $r=(2k)^d$ functions, $k^d$ of which have  smooth and compactly supported Fourier transforms, and such that $V(\Phi)$ is translation-invariant. Moreover all the generators can be chosen so that their Fourier transforms are in $H^{\frac 12-\epsilon}(\R^d)$ for all $\epsilon>0$.
\end{proposition}

We also have the following result, 

\begin{proposition}\label{prop:pointwise}
For $d\geq 1$, let $\Gamma$($>\Z^d$) be a lattice of $\R^d$ and $r\ge1$. Then there exists an SIS $V(\Phi)$ in $L^2(\R^d)$ generated by $r$ orthonormal generators $\phi_i$'s, all of which are in $L^1$ (hence they have continuous Fourier transforms) and satisfy $\omega^{\frac 12} \widehat\phi_i(\omega)\in L^\infty(\R^d)$. Moreover, $V(\Phi)$ is $\Gamma$-invariant. 
\end{proposition} 
To summarize, while Proposition \ref{prop:regular} states that it is possible to construct translation-invariant SIS with a portion of the generators having smooth and compactly supported Fourier transforms, Proposition \ref{prop:pointwise} shows that we can construct $\Gamma$-invariant SIS with all its generators having certain pointwise decay in Fourier domain.

\subsection*{Organization}
In the following section, we state and prove a convenient characterization (similar to the one given in \cite{ACHKM10}) of SIS with additional invariance.  

In Section \ref{Section:Frame}, we refine this characterization under the assumption that the generators and their translates form a frame.

In  Section \ref{section:continuous}, we prove a useful property of the Gramian under the sole assumption that the generators have continuous Fourier transform.

Sections \ref{section:mainresults} and \ref{section:mainpropositions} are dedicated to the proof of the results stated in the introduction.

\section{Finitely generated shift-invariant spaces with additional invariance}

Given a closed subgroup $\Lambda$ of $\R^d$, we define its Fourier transform $\Lambda^*$ as the closed subgroup of $\R^d$ defined by $$\Lambda^*=\{x\in \R^d, e^{2\pi i\langle \lambda ,x\rangle}=1 \; \forall \lambda\in \Lambda\}.$$
Note that the map $\Lambda\to \Lambda^*$ is an involution. The Fourier transform reverses the inclusions, namely if $\Lambda<\Gamma$ then $\Gamma^*< \Lambda^*$. 
Observe for instance that for $d=1$, the Fourier transform of $n\Z$ is $\frac{1}{n}\Z$, or that for $d=2$, the Fourier transform of $\R\times \{0\}$ is $\{0\}\times \R$.

In this paper, we consider the following general setting: $\Lambda <\Gamma$ are two closed subgroups of $\R^n$.
Let $\phi_1,\ldots, \phi_r\in L^2(\R^d)$, and denote by $\Phi$ the column vector whose components are the $\phi_i$'s. We will denote by $V^{\Lambda}(\Phi)$ the smallest closed $\Lambda$-invariant subspace containing the $\phi_i$'s. 
It is common to call $V^{\Lambda}(\Phi)$ {\it shift-invariant} when $\Lambda=\Z^d$, and {\it translation-invariant} when $\Lambda=\R^d$. To allege notation, we will omit the subscript $\Z^d$, when $\Lambda=\Z^d$.

We will be interested in the situation where $V^{\Lambda}(\Phi)$ is in addition $\Gamma$-invariant. As we will see later, this prevents frame generators from having nice decay at infinity. 
We will start by providing a short and self-contained proof of a result essentially due to \cite{ACHKM10, ACP10}. Observe that this problem is only non-trivial when the quotient $\R^d/\Lambda$ is compact since  otherwise, the only finitely generated $\Lambda$-invariant space is $\{0\}$.  This is equivalent to the fact that $\Lambda^*$ (hence $\Gamma^*$) is discrete.

In this paper we will mainly focus on the cases  when $\Lambda$ is a lattice and when $\Gamma$ is either a (larger) lattice or all of $\R^d$.

Before stating this result, let us introduce some notation. The Gramian associated to $\Phi$ and $\Lambda$ is a measurable field of $r\times r$ matrices  whose general coefficient is defined for a.e. $\omega\in \R$ by
$$G^{\Lambda}_{i,j}(\omega)= \sum_{l \in \Lambda^*} \widehat{\phi_i}(\omega+l)\overline{\widehat{\phi_j}}(\omega+l).$$
For short, $$G^{\Lambda}(\omega)= \sum_{l\in \Lambda^*}\widehat{\Phi}(\omega+l)\widehat{\Phi}^*(\omega+l).$$
For a.e. $\omega\in \R$, let
$A(\omega)$ to be the $r\times r$ matrix defined by 
$$A(\omega)=  \sum_{g\in \Gamma^*}\widehat{\Phi}(\omega+g)\widehat{\Phi}^*(\omega+g).$$

Now, let $F$ be a subset of $\Lambda^*$ consisting of representatives of the quotient $\Lambda^*/ \Gamma^*$. For instance, for $d=1$, $\Lambda=\Z$ and $\Gamma=\frac{1}{n}\Z$, one has $\Lambda^*=\Z$ and  $\Gamma^*=n\Z$, and for $F$, one can take $\{0,\ldots n-1\}$.

We have 
\begin{equation}\label{eq:GA}
G^{\Lambda}(\omega)=\sum_{f\in F}A(\omega+f).
\end{equation}
Now let us state the main result of this section. Although it could easily be deduced from the main results in \cite{ACHKM10, ACP10}, we chose to write down a short self-contained proof.

\begin{theorem}\label{thm:rank}
Let $\Lambda<\Gamma$ be closed cocompact subgroups of $\R^d$.
The space $V^{\Lambda}(\Phi)$ is $\Gamma$-invariant if and only if the following equality holds for a.e. $\omega\in \R^d$.
\begin{equation}\label{eq:rank}
\mathrm{rank\ } G^{\Lambda}(\omega)=\sum_{f\in F} \mathrm{rank\ }A(\omega+f).
\end{equation}
\end{theorem}
\begin{proof} 
Recall that a function $\varphi$ belongs to $V^{\Lambda}(\Phi)$ if and only if there exist bounded measurable $\Lambda^*$-periodic functions $P_1,\ldots,P_r$ such that $\widehat{\varphi}=P_1\widehat{\phi_1}+\ldots+P_r\widehat{\phi_r}$. Fiberwise, this is equivalent to saying that for a.e. $\omega$, the vector $(\widehat{\varphi}(\omega+l))_{l\in \Lambda^*}\in \ell^2(\Lambda^*)$ lies in the subspace $\widehat{V}(\omega)$ spanned by the $r$ vectors $(\widehat{\phi_i}(\omega+l))_{l\in \Lambda^*}$ for $i=1,\ldots ,r$. 

Now the space  $V^{\Lambda}(\Phi)$ is $\Gamma$-invariant if for every bounded $\Gamma^*$-periodic function $Q$, and for every $i=1,\ldots ,r$, there are bounded measurable $\Lambda^*$-periodic functions $P_{i,1},\ldots,P_{i,r}$ such that $$Q\widehat{\phi_i}=P_{i,1}\widehat{\phi_1}+\ldots+P_{i,r}\widehat{\phi_r}.$$
Again, this is equivalent to saying that for a.e. $\omega$, and for every bounded $\Gamma^*$-periodic $\theta: \; \Lambda^*\to \C$  the vector $(\theta(\omega+l)\widehat{\phi_i}(\omega+l))_{l\in \Lambda^*}\in \ell^2(\Lambda^*)$ lies in the subspace $\widehat{V}(\omega)$. 

Let us denote $\widehat{W}(\omega)$ the subspace of $ \ell^2(\Lambda^*)$ spanned by $(\theta(\omega+l)\widehat{\phi_i}(\omega+l))_{l\in \Lambda^*}$ for all bounded $\Gamma^*$-periodic $\theta$, and every $i=1,\ldots, r$.  Clearly $\widehat{V}(\omega)$ is a subspace of $\widehat{W}(\omega)$, and the two coincide for a.e. $\omega$ exactly when  $V^{\Lambda}(\Phi)$ is $\Gamma$-invariant. 

The rest of the proof amounts to showing that the left-hand term in (\ref{eq:rank}) corresponds to the dimension of $\widehat{V}(\omega)$ (which is obvious) and that the right-hand term of the equality corresponds to the dimension of $\widehat{W}(\omega)$.
A basis for the space of bounded $\Gamma^*$-periodic functions on $\Lambda^*$ consists of the functions $(\theta_f)_{f\in F}$, where each $\theta_f(l)$ equals $1$ if $l\in f+\Gamma^*$ and $0$ elsewhere. Observe that for all $f\neq f'$, and all $\varphi, \varphi'\in W_{\widehat{\Phi}}(\omega)$, $\theta_f\varphi$ and $\theta_{f'}\varphi'$ are orthogonal.
Hence $\widehat{W}(\omega)$ decomposes as a direct sum $$\widehat{W}(\omega)=\bigoplus_{f\in F}\widehat{W}^f(\omega),$$
where $\widehat{W}^f(\omega)$ is the subspace of functions of the form $\theta_f\varphi$, for $\varphi\in \widehat{W}(\omega).$
But it comes out that  $\widehat{W}^f(\omega)$ is precisely the subspace of $\ell^2(\Lambda^*)$ spanned by $(\widehat{\phi_i}(\omega+f+g))_{g\in \Gamma^*}$, whose dimension equals the rank of $A(\omega+f)$. This finishes the proof of the theorem.
\end{proof}

Although we will be mostly interested in the case where both $\Lambda$ and $\Gamma$ are lattices, we will also use the following special case of Theorem \ref{thm:rank}, where $\Lambda$ is a lattice and $\Gamma=\R^d$ (so $\Gamma^*=\{0\}$).

\begin{corollary}\label{cor:rank}
Let $\Lambda$ be a lattice, and let $V^{\Lambda}(\Phi)$ be a $\Lambda$-invariant space generated by $\phi_1,\ldots ,\phi_r$. Then it is translation-invariant if and only if the following equality holds for a.e. $\omega\in \R^d$.
\begin{eqnarray}\label{eq:rank2}
\mathrm{rank\ } G^{\Lambda}(\omega) &=&\sum_{f\in \Lambda^*} \mathrm{rank\ }A(\omega+f)\\
                                                              & = & \left|\{f\in \Lambda^*, \Phi(\omega+f)\neq 0 \}\right|.
\end{eqnarray}
\end{corollary}

\section{Frame generators and additional invariance}\label{Section:Frame}
The main goal of this section is to reformulate Theorem \ref{thm:rank} under the additional assumption that the generators form a frame for the space $V^{\Lambda}(\Phi)$.

Recall that the family of all $\Lambda$-translates of the generators $\phi_1, \ldots, \phi_r$ form a Riesz basis if and only if there exists $s\geq 1$ such that
\begin{equation}\label{riesz.def}
s^{-1}I\le G^{\Lambda}(\omega)\le sI, \ \text{a.e.} \ \omega\in\R.
\end{equation}
In this case the functions $\phi_i$'s are called \emph{Riesz generators} for $V^{\Lambda}(\Phi)$. 

Similarly, $\{\phi_i(\cdot+\lambda)|\ \lambda\in \Lambda, i=1, \dots, r\}$ is a frame if and only  there exists $s\geq 1$ such that
\begin{equation}\label{upperlower}
s^{-1}G^{\Lambda}(\omega)\le (G^{\Lambda}(\omega))^2\le sG^{\Lambda}(\omega)
\end{equation}
for almost every $\omega\in\R^d$ (see \cite{Bow00}). Here the $\phi_i$'s are called {\it frame generators} for $V^{\Lambda}(\Phi)$. 

Let us start by an easy lemma.
Given a non-negative self-adjoint matrix $A$, denote $q_A$ its associated quadratic form, $\mu^-(A)$ its smallest non-zero eigenvalue and $k_A$ the dimension of its kernel.  Denote the unit sphere of a Euclidean space $V$ by $S_V$.

\begin{lemma}\label{lem:matrices}
Let $C=A+B$ be three $d\times d$ non-negative self-adjoint matrices such that $\rk C=\rk A+\rk B$.
Then $\mu^-(C)\leq \min\{\mu^-(A),\mu^-(B)\}$. 
\end{lemma}
\begin{proof}
Note that we can assume  without loss of generality that $C$ has full rank. Observe that $\Ker A$ and $\Ker B$ are in direct sum. By the min-max theorem,
$$\mu^-(A)=\min_{\dim V=k_A+1}\;\max_{x\in S_V}q_A(x).$$
Now let $V_0$ be a subspace minimizing the above expression. Let $x\in V_0\cap \Ker B$ of norm $1$. We have
$$\mu^-(A)\geq q_A(x)=q_A(x)+q_B(x)=q_C(x)\geq \mu^-(C).$$
We conclude since the roles of $A$ and $B$ are symmetric.
\end{proof}

The following theorem will play a central role in the sequel.

\begin{theorem}\label{thm:frame}
Let $\Lambda<\Gamma$ be closed cocompact subgroups of $\R^d$. Let $\Phi$ be finite set of generators of $V^{\Lambda}(\Phi)$. Suppose
the space $V^{\Lambda}(\Phi)$ is $\Gamma$-invariant and $\Phi$ are frame generators.
Then, there exists $s\geq 1$ such that the formula (\ref{eq:rank}) and the following inequality hold  for a.e. $\omega\in \R^d$ 
\begin{equation}\label{eq:frame}
s^{-1}A(\omega)\leq A(\omega)^2\leq sA(\omega).
\end{equation}
\end{theorem}

\begin{proof}
By Theorem \ref{thm:rank}, it is enough to prove that (\ref{upperlower}) implies (\ref{eq:frame}) (up to changing the constant $s$). 
First observe that the previous lemma can be extended (by induction) to a sum of more than two matrices. We deduce that the lower bound in (\ref{upperlower}) implies that of (\ref{eq:frame}). The equivalence between the upper bounds follows from the fact that the number of non-zero matrices in the right-hand side of (\ref{eq:GA}) is bounded by $r$.
\end{proof}

We immediately get the following corollary.

\begin{corollary}\label{cor:frame}
Let $\Lambda$ be a closed cocompact subgroups of $\R^d$. 
Assume that the space $V^{\Lambda}(\Phi)$ is translation-invariant and $\Phi$ are frame generators. Then $Tr(A(\omega))$ is not continuous, nor in $H^{\frac 12}$.
\end{corollary}
\begin{proof}
By Theorem \ref{thm:frame}, the map $\omega\to Tr(A(\omega))$ is larger than $s^{-1}$ on a set of positive (but finite) measure, and equals zero elsewhere. In particular it is not continuous and not in $H^{\frac 12}$ (by \cite{KW99}).
\end{proof}

\begin{lemma}\label{lem: constantrank}
Let $(M(x))_{x\in \R^d}$ be a continuous family of $r\times r$ non-negative self-adjoints matrices with complex coefficients. Assume that there exists $s\geq 1$ such that 
$s^{-1}M(x)\leq M^2(x)$ for all $x\in \R^d$. Then the rank of $M$ is constant.
\end{lemma}
\begin{proof}
Since $x\to M(x)$ is continuous, its rank  is lower semi-continuous.  Let $F$ be the (closed) subset where $\mathrm{rank\ }  M$ reaches its minimal value. Assume that $F$ is not open, which means that there exists a sequence $x_n\in F^c$ converging to some $x_0\in F$. It follows that the rank of $M(x_n)$ is strictly larger than the rank of $M(x_0)$, which therefore implies that  the range of $M(x_n)$ intersects non-trivially the kernel of $M(x_0)$. Let $u_n$ be a sequence of unit vectors lying in this intersection. On one hand, because $u_n$ is in the range of $M(x_n)$,  we must have $\|M(x_n)u_n\|\geq s^{-1}$. On the other hand, the continuity of $M$ together with the fact that the limit lies in the kernel implies $M(x_n)u_n\to 0$, contradiction. So $F$ is open which implies $F=\R^d$ hence the lemma. 
\end{proof}

\begin{corollary}\label{cor:frame'}
Let $\Lambda<\Gamma$ be closed cocompact subgroups of $\R^d$. 
Assume that the space $V^{\Lambda}(\Phi)$ is $\Gamma$-invariant and $\Phi$ are frame generators. Then if $A(\omega)$ is continuous, its rank is constant.
\end{corollary}
\begin{proof}The corollary follows immediately from Lemma \ref{lem: constantrank} and Theorem \ref{thm:frame}.
\end{proof}

\section{Properties of the Gramian when the $\widehat{\phi}_i$'s are continuous}\label{section:continuous}

Although this section mainly serves as preparation for Theorem \ref{epsilonthm}, we believe that it is of independent interest and could be useful elsewhere.

From now on, $\Lambda$ will always denote a lattice in $\R^d$.
The following statement is essentially trivial (and was observed for instance in \cite{Bow00, DDR94}). 
\begin{proposition}\label{lem:min}
Let $\Lambda$ be a lattice in $\R^d$, and 
let $\Phi=(\phi_1,\ldots, \phi_r)$ be a generating set  for the $\Lambda$-invariant space $V^{\Lambda}(\Phi)$. Then the minimal number $\rho$ of generators of $V^{\Lambda}(\Phi)$ equals the essential supremum of  $\rk G^{\Lambda}(\omega)$. Equivalently, if $G^{\Lambda}(\omega)$ is non invertible a.e., then one can find a generating set $\Phi'$ with $r-1$ generators. 
\end{proposition}
\begin{proof}
This is an essentially trivial statement. Let $K$ be a fundamental domain for the action of $\Lambda^*$ (for example take $K=[0,1)^d$ if $\Lambda^*=\Z^d$). For every $\omega\in  K$, let $V^{\Lambda}(\omega)$ be the subspace of $\ell^2(\Lambda^*)$ spanned by the $r$ vectors $v_i(\omega)=\widehat{\phi_i}(\omega+l)_{l\in \Lambda^*}$ for $i=1,\ldots,r$. Saying that $G^{\Lambda}$ is non-invertible a.e. amounts to the fact that $\dim V^{\Lambda}(\omega)\leq r-1$ for a.e. $\omega\in  K$. The idea is to remove for a.e. $\omega$ one vector $v_i(\omega)$ and to recombine the other ones in order to get a new set of $r-1$ generators for $V^{\Lambda}(\Phi)$. Now, in order to do this in a measurable way, we pick the $v_i(\omega)$ of minimal index with the property that it lies in the vector space spanned by the other ones. Then we relabel the remaining ones respecting their order: for instance if $v_2$ is removed, then $v_1$ becomes $v_1'$, $v_3$ becomes $v_2'$ and so on. Now since every point of $\R^d$ can be written as $\omega+l$, for some unique $(\omega,l)\in K\times \Lambda^*$, we can define our new $r-1$ generators by the formula $\phi'_i(\omega+l)=(v'_i(\omega))_l$, for all $l\in \Lambda^*$.  
\end{proof}

We deduce from this proposition that the essential supremum of $\rk G(\omega)$ equals $\rho$.

\begin{lemma}\label{lem:semicont}
Suppose that the functions $\widehat{\phi}_i$ are continuous for $i=1, \dots, r$. Then the maps 
$\omega\to \rk A(\omega)$ and $\omega\to \rk G^{\Lambda}(\omega)$ from $\R^d$ to $\N$ are lower semi-continuous. 
\end{lemma}
\begin{proof}
Since $G^{\Lambda}(\omega)=\sum_F A(\omega+f)$ it is enough to show that $\rk A(\omega)$ is lower semi-continuous. 
Observe that $A(\omega)$ is the sum over $g\in \Gamma^*$ of the continuous positive semi-definite matrices $\widehat{\Phi}(\omega+g) \widehat{\Phi}(\omega+g)^*$. Therefore the rank of $A(\omega)$ is the supremum of the ranks of all partial finite sums. Since lower semi-continuity is stable under taking supremums, we deduce that $\rk A(\omega)$ is lower semi-continuous.
\end{proof}

\begin{proposition}\label{prop:det}
 Suppose that the functions $\widehat{\phi}_i$ are continuous, and that $V^{\Lambda}(\Phi)$ is $\Gamma$-invariant for some lattice $\Gamma$ such that $[\Gamma:\Lambda]$ does not divide $\rho$.  Then the subset $\{\omega, \; \rk G^{\Lambda}(\omega)<\rho\}$ is a non-empty closed subset of $\R^d$. 
\end{proposition}
\begin{proof}
The fact that $\{\omega, \; \rk G^{\Lambda}(\omega)\leq \rho-1\}$ is closed results from Lemma \ref{lem:semicont}. We therefore only have to prove that it is non-empty. Let us assume on the contrary that $G^{\Lambda}(\omega)$ has rank $\rho$ for all $\omega$.  
We know by Theorem \ref{thm:rank} that for a.e. $\omega\in \R^d$, 
\begin{equation}\label{eq:r}
\rho=\sum_{f\in F} \rk A(\omega+f).
\end{equation}
Moreover the equality is an inequality $\leq$ for all $\omega$. 
But again Lemma \ref{lem:semicont} implies that the $\omega$'s for which the inequality is strict form an open set, which therefore has to be empty.

Now observe that lower semi-continuous functions with integer values can only increase locally. In other words for every $\omega_0$ the set $$\{\omega, \; \rk A(\omega)\geq \rk A(\omega_0)\}=\{\omega, \; \rk A(\omega) > \rk A(\omega_0)-1/2\}$$ is a non-empty open set.  Together with (\ref{eq:r}), this immediatly implies that $\rk A(\omega)$ is locally constant, hence constant on $\R^d$ (since $\R^d$ is connected!). Let us call $m$ the corresponding integer. We therefore have $\rho=m[\Gamma:\Lambda]$, contradiction.
\end{proof}

\section{Proofs of The main results}\label{section:mainresults}

\subsection{Proof of  Theorem \ref{notranslation}}
Immediately follows from Corollary \ref{cor:frame}. 

\subsection{Proof of  Theorem \ref{pointwise}}
Suppose by contradiction that there exists $\epsilon>0$ such that $\omega^{\frac d2+\epsilon} \widehat\phi_{i}(\omega)\in L^\infty(\R^d)$ for all $i$. This easily implies that $A(\omega)$ is continuous, so we conclude by Corollary \ref{cor:frame'}: namely since $\rk A(\omega)$ is constant it must divide the rank of $G(\omega)$, which by (\ref{eq:rank})  would imply that the index of $\Lambda$ in $\Gamma$ divides $r$.

\subsection{Proof of Theorem \ref{epsilonthm}}
The proof of Theorem \ref{epsilonthm} relies on the following result of harmonic analysis. Its proof for $d=1$ is essentially the proof of
\cite[Theorem1.2]{ASW11}. Since it extends without change to any $d$, we do not reproduce it here.  
\begin{lemma}\label{epsilonlemma}
Let $\Lambda$ be a lattice in $\R^d.$
Suppose that a function $f\in L^2(\R^d)$ satisfies the following properties
\begin{enumerate}
\item there exists $\epsilon>0$ such that
$$\int_{\R^d} |f(x)|^2 |x|^{d+\epsilon} dx<\infty.$$

\item There exists a constant $C$ such that for a.e. $\omega\in \R^d$,
\begin{equation}\label{upperbond}
\sum_{k\in\Lambda^*}|\widehat f(\omega+k)|^2\le C.
\end{equation}
\item There exists $\omega_0$ is such that 
$$\widehat{f}(\omega_0+k)=0 \quad {\rm for \ all}\  k\in\Lambda^*,$$
\end{enumerate}
Then for all $\eta>0$, there exists $\delta_0>0$ only depending on $\eta,\epsilon$ and $C$ such that for all $0<\delta\leq \delta_0$
$$\frac{1}{\delta^d}\int_{B(\omega_0,\delta)} \sum_{k\in\Lambda^*}|\widehat{f}(\omega+k)|^2d\omega<\eta.$$
\end{lemma}

\begin{proof}[Proof of Theorem \ref{epsilonthm}] We shall use the notation introduced in Section \ref{Section:Frame} (before stating Lemma \ref{lem:matrices}). 
Let us suppose by contraction that there exits $\epsilon>0$ such that 
\begin{equation}\label{epsilonequ}
\int_{\mathbb R^d} |\phi_{i}(x)|^2 |x|^{d+\epsilon} dx<\infty,
\end{equation}
for every $i$. Note that this condition implies that  the $\phi_i$ are in $L^1$, so that their Fourier transforms are uniformly continuous. 

Define $E=\{\omega,\;\rk G^{\Lambda}(\omega)<\rho\}$, which by Proposition \ref{prop:det}, is a non-empty closed subset of $\R^d$. Take an open ball $B\subset E^c$ whose boundary interests $E$ in a point $\omega_0$. Observe that at least one fourth of the volume of a  sufficiently small ball centered in $\omega_0$  is contained in $E^c$. We will use this remark at the end of the proof.

Since $\omega_0\in E$, there exists an $(r-\rho+1)$-dimensional subspace $V_0$ of $\R^r$ on which $G^{\Lambda}(\omega_0)$ vanishes. Since $G^{\Lambda}(\omega)=F(\omega)F^*(\omega)$ where $F(\omega)$ is the $r\times |\Lambda^*|$ matrix whose $(i,l)$-coefficient is  $F(\omega)_{il}=\widehat{\phi}_i(\omega+l)$, we have $F^*(\omega_0)a=0$ for all $a\in V_0$. This implies that 
$$\sum_{i=1}^r a_i\widehat{\phi}_i(\omega_0+l)=0 \quad {\rm for \ all}\ l\in\Lambda^*.$$
In other words, the function $f_{a}(x)=\sum_{i=1}^ra_i\phi_i(x)$ satisfies $\widehat{f_{a}}(\omega_0+l)=0$ for all $l\in \Lambda^*$ and for all $a\in V_0$. 
Now clearly the functions $f_a$ for $a\in S_{V_0}$ satisfy the assumptions of the lemma with $C$ independent of $a.$ We therefore get for $\delta>0$ small, 
\begin{equation}\label{eq}
\frac{1}{\delta^d}\int_{B(\omega_0,\delta)} \sum_{k\in\widehat{\Lambda}}|\widehat{f_a}(\omega+k)|^2d\omega<\eta.
\end{equation}

Suppose $a_1, \dots, a_{r-\rho+1}$ is an orthonormal basis of $V_0$. Then for a.e. $\omega\in B$, the frame condition implies 
$$s^{-1}\leq \mu^{-}(G^{\Lambda}(\omega))\leq \sum_{i=1}^{r-\rho+1}\langle G^{\Lambda}(\omega)a_i,a_i\rangle.$$    
By the above remark, we see that this contradicts (\ref{eq}) provided $\eta$ is small enough.
\end{proof}

\section{Proofs of Proposition \ref{prop:regular} and Proposition \ref{prop:pointwise}}\label{section:mainpropositions} 
\subsection{Proof of Proposition \ref{prop:regular}}
It is actually enough to provide a construction for $d=1$. Indeed given such $\phi_1,\ldots,\phi_{2k}$ in $L^2(\R)$, we define generators in dimension $d$ by considering tensor products of these functions: namely $\tilde{\phi}_{i_1,\ldots,i_d}=\phi_{i_1}\otimes \ldots\otimes\phi_{i_d}.$ Then we get that $V_{\tilde{\Phi}}$ is simply the tensor product $V_{\Phi}^{\otimes d}$, whose generators are clearly orthonormal. Translation invariance of $V_{\tilde{\Phi}}$ therefore follows from that of $V_{\Phi}$.

Hence, let us focus on the one-dimensional case.
Let $g$ be an infinitely-differentiable function  that satisfies $g(x)=0$ when $x\le 0$, $g(x)=1$ when $x\ge 1$,
 and $g^2(x)+g^2(1-x)=1$ when $0\le x\le 1$.  Define $\widehat{\phi}_1(\omega)=g(\omega)g(2-\omega)$ and $\widehat{\phi}_2(\omega)=g(1-\omega)\chi_{[0,1)}-g(\omega-1)\chi_{[1,2)}$.

Then for $i\le r$, define $\widehat{\phi}_i(\omega)=\widehat{\phi}_1(\omega-i+1)$ when $i$ is odd and $\widehat{\phi}_i(\omega)=\widehat{\phi}_2(\omega-i+2)$ when $i$ is even ($\widehat{\phi}_1$ and $\widehat{\phi}_2$ are plotted in Figure 1). It is easy to check that $\sum_{l\in \Z}|\widehat{\phi}_i(\omega+l)|^2=1$ a.e. $\omega$ for all $i$, and that $\sum_{l\in \Z}\widehat{\phi}_1(\omega+l)\overline{\widehat{\phi}_2}(\omega+l)=0$. We deduce that $\sum_{l\in \Z}\widehat{\phi}_i(\omega+l)\overline{\widehat{\phi}_j}(\omega+l)=0$ for $i\neq j$. This shows that $G(\omega)= I$ which amounts to saying that the functions $\phi_i$ are orthonormal generators for $V(\Phi)$. 

 \begin{figure}[hbt]
\centering
\begin{tabular}{c}
  \includegraphics[width=80mm]{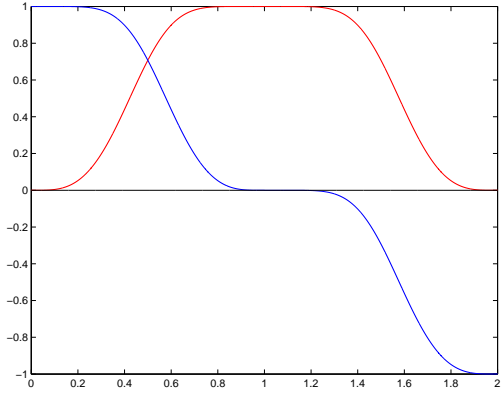} 
 
\end{tabular}
\caption{The function $\widehat{\phi}_1$ being red and the function $\widehat{\phi}_2$ being blue.
}
 \label{figure1}
\end{figure}

Since all the functions are supported on $[0,r]$, we can show that, for a.e. $\omega\in[0,1)$, $\mathrm{rank\ }A(\omega+f)=1$ when $0\le f\le r-1$ and $\mathrm{rank\ }A(\omega+f)=0$ elsewhere. Hence $\sum_{f\in \Z}\mathrm{rank\ }A(\omega+f)=r=\mathrm{rank\ } G^{\Lambda}(\omega)$. From Corollary \ref{cor:rank}, $V(\Phi)$ is translation-invariant.  

Observe that when $i$ is odd, $\widehat{\phi}_i$ is compactly supported and infinitely-differentiable. When $i$ is even, $\widehat{\phi}_i$ can be written as a product of a compactly supported and infinitely-differentiable function and the characteristic function $\chi_{[i-2,i)}$. Since $\chi_{[i-2,i)}$ belongs to $H^{\frac12-\epsilon}$, we have $\widehat{\phi}_i\in H^{\frac12-\epsilon}$ for any $\epsilon>0$.

\subsection{Proof of Proposition \ref{prop:pointwise}}
We assume that $\Gamma=\frac{1}{n_1}\Z\times\frac{1}{n_2}\Z\times\cdots\times\frac{1}{n_d}\Z$. 
We first need the following lemma which treats the case when $d=1$ and $r=1$.
\begin{lemma}[\cite{ASW11}]\label{dimension1}
For integer $n_1$, there exists a function $\phi_1 \in L^1\cap L^2$  (and hence $\widehat \phi_1$ is continuous),  such that $\phi_1$ is an orthonormal generator for its generating space $V(\phi_1)$,
$V(\phi_1)$ is $\frac{1}{n_1}\Z$-invariant and $\omega^{\frac 12} \widehat\phi_1(\omega)\in L^\infty(\R)$. 
\end{lemma}

Since our method relies on the construction of $\phi_1$ in Lemma \ref{dimension1}, we will describe $\phi_1$ explicitely here. Let $g$ be the function as in the proof of Proposition  \ref{prop:regular}. Define $ g_0(x)= g(x+1) g(-x+1)$ and $g_1(x)=g(x+1)g(-2x+1)$. Then the Fourier transform of $\phi_{1_1}$ is defined to be 

\begin{eqnarray}\label{time.lem.eq1}
\widehat {\phi_1}(\omega) & = &  h_0(\omega)+\sum_{j=1}^\infty \sum_{l=0}^{4^j-1}  2^{-j} h_j(\omega-n_1 (\gamma_j+l))\nonumber\\
& & \quad +\sum_{j=1}^{\infty} \sum_{l=0}^{4^j-1} 2^{-j}
h_j(-\omega-n_1 (\gamma_{j}+l)), 
\end{eqnarray}
where $\gamma_j=\sum_{k=0}^{j-1} 4^k$, $h_0(\omega)=g_0(4\omega)$ and $h_j(\omega)=g_1(2^{j+1}\omega-2^j+1)$. 

It is not hard to check that $\widehat {\phi_1}$ is an even function such that $|\widehat\phi_1|\le1$ and has support on $[-\frac14,\frac14]\cup( \cup_{j=1}^\infty E_{1j}\cup E_{1j}^\prime)$, where $E_{1j}=\cup_{l=0}^{4^j-1}[\frac12-\frac1{2^j}+n_1(\gamma_j+l), \frac12-\frac1{2^{j+2}}+n_1(\gamma_j+l)]$ and $E_{1j}^\prime=\cup_{l=0}^{4^j-1}[\frac1{2^j}-\frac12-n_1(\gamma_j+l), \frac1{2^{j+2}}-\frac12-n_1(\gamma_j+l)]$.

\noindent{\bf From one to several generators (in dimension 1).} Now we want to construct a $\frac{1}{n_1}\Z$-invariant SIS with $r$ orthonormal generators in $L^1\cap L^2$ satisfying our pointwise decay property in Fourier domain.

Define $\phi_{i}$ for $2\le i\le r$ to be such that $\widehat {\phi}_{i}(\omega)=0$ when $\omega\in[0,n_1\gamma_{2i-2}-1]$ and 
\begin{align}
\widehat {\phi}_{i}(\omega)= \frac{1}{\sqrt2}\Big(\sum_{l=0}^{4^{2i-2}-1}2^{-(2i-2)}h_0(\omega-n_1(\gamma_{2i-2}+l))\Big)\nonumber\\
+\sum_{j=1}^\infty \sum_{l=0}^{4^{j+2i-2}-1}  2^{-(j+2i-2)} h_j(\omega-n_1(\gamma_{j+2i-2}+l))
\end{align}

when $\omega\ge n_1\gamma_{2i-2}-1$, and $\widehat \phi_{i}(\omega)= \widehat \phi_{i}(-\omega)$ when $\omega\le0$. 

It is easy to see that each $\widehat\phi_{i}$ obeys the pointwise decay property and that it is an orthonormal generator for the principal SIS $V(\phi_{i})$. In fact, 
$$\sum_{l\in {\mathbb Z}}|\widehat \phi_{i}(\omega+l)|^2=\sum_{l\in {\mathbb Z}}|\widehat \phi(\omega+l)|^2=1$$ 
for a.e. $\omega$. One can also check that for $2\le i\le r$, $\widehat\phi_{i}$ has support 
\begin{align}
\Big( \cup_{j=1}^\infty E_{ij}\cup E_{ij}^\prime\Big)\bigcup\Big(\cup_{l=0}^{4^{2i-2}-1}[-\frac14-n_1(\gamma_{2i-2}+l),\frac14-n_1(\gamma_{2i-2}+l)]\Big)\nonumber\\
\bigcup\Big(\cup_{l=0}^{4^{2i-2}-1}[-\frac14+n_1(\gamma_{2i-2}+l),\frac14+n_1(\gamma_{2i-2}+l)]\Big),
\end{align}
where $E_{ij}=\cup_{l=0}^{4^{j+2i-2}-1}[\frac12-\frac1{2^j}+n_1(\gamma_{j+2i-2}+l), \frac12-\frac1{2^{j+2}}+n_1(\gamma_{j+2i-2}+l)]$ and $E_{ij}^\prime=\cup_{l=0}^{4^{j+2i-2}-1}[\frac1{2^j}-\frac12-n_1(\gamma_{j+2i-2}+l), \frac1{2^{j+2}}-\frac12-n_1(\gamma_{j+2i-2}+l)]$. 

In particular,  each $V(\phi_{i})$ is $\frac1{n_1}$-invariant and all the $\phi_{i}$'s have disjoint support in Fourier domain. Hence $V(\Phi_1)$ is also $\frac1{n_1}$-invariant, where $\Phi_1$ is the column vector whose components are the $\phi_{i}$'s. In fact
\begin{equation}\label{orthognal}
V(\Phi_1)=\bigoplus_{i\le r}V(\phi_{i})
\end{equation}

From Lemma \ref{dimension1}, for each $j\ge2$, we can construct $V(\psi_j)$ with orthonormal generator $\psi_j$ having the desired properties and that $V(\psi_j)$ is $\frac1{n_j}\Z$-invariant.  

\noindent{\bf Higher dimension.} 
Like in the previous section, we let  $\tilde{\phi}_{i}$ be the $d$-fold tensor product $\phi_{i}\otimes\psi_2 \ldots\otimes\psi_d$, for $i=1,\ldots, r$. By construction, the shift-invariant space generated by these $r$ orthonormal function is $\Gamma$-invariant. Since all $\phi_{i}$'s and $\psi_{j}$'s are in $L^1(\R)$, it follows that the $\tilde{\phi}_{i}$'s are in $L^1(\R^d)$. Pointwise decay results from the fact that all these functions have bounded Fourier transforms.

\bigskip

{\bf Acknowledgement} \quad The authors would like to thank Akram Aldroubi and Qiyu Sun for valuable discussions. The authors would also like to thank Carolina Mosquera and Victoria Paternostro for their comments on the proof of Theorem 1.2.

\end{document}